\documentclass[12pt]{amsart}
\usepackage{amsmath}
\usepackage{amsthm}
\usepackage{thmtools}
\usepackage{amssymb}
\usepackage[dvipsnames,svgnames,table]{xcolor}
\usepackage[paper=a4paper,left=25mm,right=25mm,top=25mm,bottom=25mm]{geometry}
\usepackage[colorlinks=true,linkcolor=RoyalBlue,citecolor=PineGreen,urlcolor=RoyalBlue]{hyperref}
\usepackage[nameinlink]{cleveref}
\usepackage[
    backend=biber,
    maxnames=4,
    maxalphanames=4,
    style=alphabetic,
    backref
  ]{biblatex}

\usepackage{tikz}
\usetikzlibrary{calc}

\crefformat{equation}{#2(#1)#3}
\Crefformat{equation}{#2(#1)#3}
  
\newcommand{\Z}{\mathbb Z}
\newcommand{\F}{\mathbb F}

\newtheorem{lemma}{Lemma}
\newtheorem{theorem}{Theorem}

\newtheorem{proposition}{Proposition}

\title{Large solution-free sets via combinatorial degenerations}
\author{Paul Hametner}
\address{Institute of Science and Technology, Austria}
\email{paul.hametner@ista.ac.at}

\author{J\'ozsef Solymosi}
\address{Department of Mathematics, University of British Columbia, Vancouver, BC, Canada; and \'Obuda University, Budapest, Hungary}
\email{solymosi@math.ubc.ca}

\begin{document}

\begin{abstract}
    Consider the linear form \[L(x,y,z,w) = 3x + y - 2z-2w.\]
    For a positive integer $N$, denote by $r_L(N)$ the largest size of a subset of $\{1,2,\dots,N\}$ that avoids nontrivial solutions to $L = 0$. We show that $r_L(N) = \Omega(N^{0.5608687})$, improving the lower bound for Problem 16 in Green's list of open problems. Our proof uses the method of combinatorial degenerations to turn a finite certificate into large solution-free sets. In fact, we can improve Ruzsa's lower bound of $N^{1/2-o(1)}$ for many four-variable equations. Consider
    \[L(x,y,z,w) = ax+by-cz-dw\]
    with $a,b,c,d\in \Z_{>0}$, $a+b=c+d$, $\{a,b\} \neq \{c,d\}$ and $abcd$ not a square. We show that there is $\varepsilon_L > 0$ such that $r_L(N) = \Omega_L(N^{1/2 + \varepsilon_L})$. Furthermore, we show that for primitive translation-invariant linear forms $L$ in $s$ variables, the lower bound $\Omega_s(N^{1/(s-1)})$ coming from a greedy construction is never optimal.
\end{abstract}

\maketitle

\section{Introduction}
A central paradigm in additive combinatorics is the study of sets of integers, or, more generally, subsets of abelian groups, that do not have solutions to a linear equation. In this paper, we are interested in finding large solution-free sets. Consider a translation-invariant linear form $L$, that is,
\begin{align}\label{align:L_def}
L(x_1,\dots,x_s) = c_1x_1+\dots + c_s x_s,
\end{align}
with $c_1,\dots,c_s \in \Z\setminus \{0\}$ and $c_1+\dots + c_s = 0$. Let $G$ be an abelian group. For a solution $(x_1,\dots, x_s) \in G^s$ to $L=0$, consider the partition $\mathcal T$ of $[s]$ with the property that $i, j \in [s]$ are in the same block of $\mathcal T$ if and only if $x_i = x_j$. We call $(x_1,\dots,x_s)$ a \emph{trivial} solution to $L=0$ if $\mathcal T$ satisfies
\begin{align}\label{align:coeff_partition}
    \sum_{i \in T} c_i = 0, \quad \text{for all } T \in \mathcal T,
\end{align}
and \emph{nontrivial} otherwise. Note that here we interpret the coefficients $c_i$ as endomorphisms of $G$, so that this notion of triviality depends on $G$. In \cite{RuzsaSolving}, Ruzsa started the systematic study of how large subsets of $[N] := \{1,2,\dots,N\}$ avoiding nontrivial solutions to $L=0$ can be. We call a set $A \subseteq G$ $L$-free if all solutions $(x_1,\dots,x_s)\in A^s$ to $L=0$ are trivial, and we let
\[r_L(N) = \max\{|A| : A \subseteq [N] \text{ is $L$-free}\}.\]
Up to a factor of $N^{o(1)}$, the current best upper bound for $r_L$ is Ruzsa's genus bound. Define the genus of $L$ to be the largest number of blocks of a partition $\mathcal T$ with property \Cref{align:coeff_partition} and denote it by $g$. Ruzsa showed that
\footnote{For functions $f,g: \Z_{> 0} \to \Z_{> 0}$, we write $f = O(g)$, equivalently $g = \Omega(f)$, if $f \leq Cg$ for a constant $C > 0$. Parameters that $C$ might depend on may appear in the subscript of $O$ and $\Omega$. A quantity that goes to zero as $N\to \infty$ is denoted by $o(1)$.}
\[r_L(N) = O_L(N^{1/g})\]
and the central question in the area is whether $r_L(N) \geq N^{1/g-o(1)}$. The simplest forms $L$ for which the exponents in the best-known upper and lower bounds for $r_L(N)$ do not match are four-variable translation-invariant forms 

\begin{align}\label{align:four_vars_form}
    L_4(x,y,z,w) =ax+by-cz-dw,
\end{align}

with $a,b,c,d\in \Z_{>0}$ and $\{a,b\} \neq \{c,d\}$. For these, Ruzsa used a parabola mod $p$ to show that $r_{L_4}(N) \geq N^{1/2 - o(1)}$ and, when $abcd$ is not a square, $r_{L_{4}}(N) = \Omega(\sqrt N)$. The form 
\[L_3(x,y,z,w) = 3x+y-2z-2w\]
has attracted particular interest. Ruzsa \cite{RuzsaSolving} asked if one can get a better lower bound than $\Omega(\sqrt N)$ for $r_{L_3}(N)$, and the question also appears in \cite{green100problems} as Problem 16.

\begin{theorem}\label{theorem:3x+y=2z+2w}
    Let $L_3(x,y,z,w) = 3x+y-2z-2w$. Then 
    \[r_{L_3}(N) \geq N^{\alpha-o(1)}, \quad \text{where } \alpha =\frac{\log15}{3\log 5} > 0.5608687.\]
\end{theorem}

The proof of \Cref{theorem:3x+y=2z+2w} relies on a finite certificate of $15$ points in $\F_5^3$, which is converted into large $L_3$-free sets in $[N]$ using \Cref{theorem:main} in \Cref{section:degenerations}. The main ingredients of \Cref{theorem:main} are degenerations, which trace back to Strassen \cite{Strassen1987}. Instead of the $M$-degenerations that Strassen introduces for tensors, we use the closely related combinatorial degenerations introduced in \cite[Definition 15.29]{algebraicComplTheory}. In \cite{degenerations_corners}, combinatorial degenerations are used to lower-bound the Shannon capacity of directed hypergraphs, which is one way to see our degeneration lemma, \Cref{lemma:L-seed}. For notational convenience, we use the word $L$-seed instead of combinatorial degeneration. 

\medskip

\Cref{theorem:main} works more generally for primitive forms. A linear form $L$ as in \Cref{align:L_def} is primitive if there is a unique finest partition with the property \Cref{align:coeff_partition}. The key property of primitive forms $L$ is that $L$-freeness is preserved under taking Cartesian products, that is, if $A$ and $B$ are $L$-free, then so is $A \times B$. We will use this observation throughout the paper.

\medskip

Furthermore, we are able to improve Ruzsa's parabola construction for every four-variable form $L_4$ as in \Cref{align:four_vars_form} with $abcd$ not a square. 
\begin{restatable}{theorem}{fourvars}\label{theorem:four_vars_form}
    Let $L_4(x,y,z,w) = ax+by-cz-dw$ be a translation-invariant form with $a,b,c,d \in \Z_{>0}$ and $abcd$ not a square. Then 
    \[r_{L_4}(N) = \Omega(N^{1/2+\varepsilon_{L_4}}),\quad \varepsilon_{L_4} > 0.\]
\end{restatable}

In the proof, we do not use a finite certificate like the one used in the proof of \Cref{theorem:3x+y=2z+2w}, but rather apply a recent result of the first author and Tyrrell \cite{beatingProducts} to Ruzsa's parabola construction, vitally exploiting that $L_4$ has genus one. 

\medskip

For a general $s$-variable translation-invariant form $L(x_1,\dots,x_s) = \sum_{i = 1}^s c_i x_i$, Ruzsa observed a lower bound of $r_L(N) = \Omega_s(N^{1/(s-1)})$ using a greedy construction. If $L$ is primitive, the moment curve in $\F_p^{s-1}$ yields the same bound and, in fact, can be improved by taking the projective completion of the moment curve. 

\begin{theorem}\label{theorem:proj_compl_moment}
For every primitive $L$ with $s \geq 3$, there is $\varepsilon_L>0$ such that 
\[r_L(N) = \Omega_{L}\left( N^{\frac{1}{s-1}+\varepsilon_L}\right).\]
\end{theorem}

Under certain conditions, we can also prove that already the moment curve in $\F_p^{s-2}$ is $L$-free. Note that the following is a direct generalization of Ruzsa's parabola construction for four-variable $L_4$ as in \Cref{align:four_vars_form} with $abcd$ not a square.

\begin{theorem}\label{theorem:reduced_moment_even_s}
    Let $L$ be primitive and let $s$ be even. If 
    \[(-1)^{s/2} \prod_{i = 1}^s c_i\]
    is not a square, then 
    \[r_L(N) = \Omega_{L}( N^{\frac1{s-2}}).\]
\end{theorem}

\medskip

The rest of the paper is organized as follows. First, we introduce the combinatorial degeneration method and prove \Cref{theorem:main} in \Cref{section:degenerations}. In \Cref{section:peeling}, we give the certificate giving the exponent in \Cref{theorem:3x+y=2z+2w} and we discuss how Behrend's construction and the peeling method can be seen as special cases of combinatorial degeneration. Solution-freeness of moment curves over $\F_p$ is discussed in \Cref{section:parabola_and_beyond}, where we prove \Cref{theorem:four_vars_form,theorem:proj_compl_moment,theorem:reduced_moment_even_s}.

\section{Combinatorial degenerations and solution-free sets}\label{section:degenerations}

Throughout the paper, we let $L(x_1,\dots,x_s) = c_1x_1 + \dots + c_s x_s$ be a translation-invariant linear form in $s\geq 3$ variables. If $L$ is primitive, the finest partition $\mathcal T$ of $[s]$ with the property 
\begin{align}\label{align:coeff_partition2}
    \sum_{i \in T} c_i = 0, \quad \text{for all } T \in \mathcal T,
\end{align}
is called the \emph{coefficient partition} of the primitive form $L$. Similarly, for an integer $m \geq 2$, we say that $L$ is \emph{primitive mod $m$} if there is a unique finest partition $\mathcal T$ so that \Cref{align:coeff_partition2} holds in $\Z_m = \Z/m\Z$. We refer to this $\mathcal T$ as the \emph{coefficient partition of $L$ mod $m$}.

\medskip

A subset $A$ of a finite abelian group $G$ is an $L$-seed if there are $f_1,\dots,f_{s} : A \to \Z$ such that 
\begin{align*}
    \sum_{i = 1}^{s} f_i(x_i) = 0 \quad &\text{for all trivial solutions } (x_1,\dots,x_s) \in A^s, \\
    \sum_{i = 1}^{s} f_i(x_i) > 0 \quad &\text{for all nontrivial solutions } (x_1,\dots,x_s) \in A^s.
\end{align*}
We call the $f_i$ the certificate functions for $A$. Note that, because of the triviality of the diagonal solutions $(x,x,\dots,x) \in G^s$, they always satisfy $f_1+\dots+f_s = 0$.

\begin{theorem}\label{theorem:main}
    Let $L$ be primitive over $\Z_m$ such that $L$ has the same coefficient partition over $\Z_m$ as over $\Z$ and let $A \subseteq \Z_m^r$ be an $L$-seed. Then 
    \[r_L(N) \geq N^{\frac1r\log_m|A|} \exp\left(-O_{A,L}\left(\sqrt{\log N \log \log N}\right)\right).\]
\end{theorem}
Conceptually, the theorem says that if we find an $L$-seed $A \subseteq G = \Z_m^r$ of size $|A| = |G|^\alpha$, then this \emph{rate} $\alpha$ essentially carries over to the integer case, meaning that $r_L(N) \geq N^{\alpha - o(1)}$. The condition of the coefficient partitions being the same over $\Z_m$ and $\Z$ is rather weak, and only finitely many $m$ do not satisfy it. 

\medskip

In the genus one case, Theorem 23 in \cite{degenerations_corners} can be applied directly to lift the $L$-seed to $L$-free sets. Their Theorem 23 works more generally for lower-bounding the Shannon capacity of directed hypergraphs. We also cover the higher genus case and use a slicing argument instead of the Salem--Spencer-type argument in their proof. 

\begin{lemma}\label{lemma:L-seed}
    Let $L$ be primitive, let $G$ be a finite abelian group and let $A \subseteq G$ be an $L$-seed with certificate functions $f_1,\dots,f_s$. Let $K \geq 0$ be such that for all $i \in [s-1]$ and all $x \in A$, we have $|f_i(x)| \leq K$. Then for every integer $\ell \geq 1$ there exists an $L$-free set $X \subseteq G^\ell$ of size $|X| \geq |A|^\ell/(1+2\ell K)^{s-1}$.
\end{lemma}
\begin{proof}
For a fixed integer vector $\kappa=(k_1,\dots,k_{s-1}) \in [-\ell K,\ell K]^{s-1}$, consider the slice 
\[X_\kappa = \left\{(x_1,\dots,x_\ell) \in A^\ell :\, \sum_{j = 1}^\ell f_i(x_j)= k_i \text{ for all } i \in [s-1]\right\}\]
and note that $X_\kappa$ is $L$-free. Indeed, suppose there is a solution $x^{(1)}, \dots,x^{(s)} \in X_\kappa$ with $L(x^{(1)}, \dots,x^{(s)}) = 0$. Write $x^{(i)} = (x_1^{(i)},\dots,x_\ell^{(i)})$ so that $L(x^{(1)}_j, \dots,x^{(s)}_j) = 0$ for all $j \in [\ell]$. Since $A$ is an $L$-seed with certificate functions $f_1,\dots,f_{s}$, we have 
\begin{align}\label{align:degen_condition_of_coord}
    \sum_{i = 1}^{s-1} f_i(x_j^{(i)}) - \sum_{i = 1}^{s-1} f_i(x_j^{(s)})=\sum_{i = 1}^{s} f_i(x_j^{(i)})\geq 0, \quad \text{ for all } j \in [\ell],
\end{align}
with equality if and only if $(x_j^{(1)}, \dots, x_j^{(s)})$ is trivial. Summing over $j$ gives
\[\sum_{j=1}^\ell\sum_{i = 1}^{s-1} f_i(x_j^{(i)}) = \sum_{i = 1}^{s-1} k_i = \sum_{j=1}^\ell\sum_{i = 1}^{s-1} f_i(x_j^{(s)})\]
so that in \Cref{align:degen_condition_of_coord} we actually have equality for every $j \in [\ell]$. This shows that $(x^{(1)},\dots,x^{(s)})$ is trivial. Now take $\kappa$ such that $|X_\kappa|$ is maximal. By the pigeonhole principle, we have $|X_\kappa| \geq |A|^\ell/(1+2\ell K)^{s-1}$. 
\end{proof}

For an element $x \in \Z_m$, we write $\widetilde x$ for its representative in $\{0,\dots,m-1\}$. If $x= (x_1,\dots,x_n) \in \Z_m^n$, then $\widetilde x = (\widetilde x_1,\dots,\widetilde x_n)$. In the next lemma, we use base expansions twice to project the $L$-free set just constructed to the integers. 

\begin{lemma}\label{lemma:project_to_integers}
    Let $L$ be primitive over $\Z_m$ such that $L$ has the same coefficient partition over $\Z_m$ as over $\Z$. Assume there is an $L$-free set $X \subseteq \Z_m^n$ of size $m^{\alpha n}$. Let $c_L = \frac12\sum_{i=1}^s |c_i|$. Then for any integer $N \geq 1$, we have
    \[r_L(N) \geq \frac1{(c_Lm)^{\alpha n}} N^{\alpha}.\]
\end{lemma}
\begin{proof}
We first show that, for an integer $d \geq 1$ that we choose later, the set
\[Y = \left\{\sum_{j = 1}^d \widetilde x_j m^{j-1}  : x_1,\dots,x_d \in X\right\} \subseteq \{0,\dots, m^d-1\}^n\]
is $L$-free. Suppose there is a nontrivial solution $(x^{(1)},\dots,x^{(s)}) \in Y^s$ to $L = 0$ and write 
\[x^{(i)} = \sum_{j = 1}^d \widetilde x_j^{(i)} m^{j-1}, \quad \text{ for } i \in [s].\]
Since $L$ is primitive, there is $\ell \in [d]$ such that $(\widetilde x_\ell^{(1)}, \dots,\widetilde x_\ell^{(s)})$ is not a trivial solution to $L = 0$ over $\Z$. Consider the smallest such index $\ell \in [d]$. This in particular guarantees that $L(\widetilde x_j^{(1)},\dots,\widetilde x_j^{(s)}) =0$ for $1 \leq j < \ell$. Taking our equation mod $m^\ell$ gives
\[0=L(x^{(1)},\dots,x^{(s)})=\sum_{j = 1}^d  L(\widetilde x_j^{(1)},\dots,\widetilde x_j^{(s)}) m^{j-1}\equiv L(\widetilde x_\ell^{(1)}, \dots, \widetilde x_\ell^{(s)})m^{\ell-1} \mod m^\ell,\]
and thus $L(x_\ell^{(1)}, \dots, x_\ell^{(s)}) =0$ in $\Z_m$. Since $X$ is $L$-free, $(x_\ell^{(1)}, \dots, x_\ell^{(s)})$ must be a trivial solution, contradicting our choice of $\ell$.\\

Now let $M = 1+c_L(m^d-1)$ and take  
\[Z = \left\{1+\sum_{k = 1}^n y_k M^{k-1}  : (y_1,\dots,y_n) \in Y\right\} \subseteq [M^n].\]
We claim that $Z$ is also $L$-free. Indeed, suppose there is a solution $(z^{(1)},\dots,z^{(s)}) \in Z^s$ to $L = 0$ and write 
\[z^{(i)} = 1+\sum_{k = 1}^n y_k^{(i)} M^{k-1}.\]
Note that, by the choice of $M$, we have $|L(y_k^{(1)}, \dots, y_k^{(s)})|<M$ for every $k \in [n]$. By the uniqueness of base-$M$ expansion, this means that $(y_k^{(1)}, \dots,y_k^{(s)})$ is a solution to $L=0$ for every $k \in [n]$. Since $Y$ is $L$-free, each of them is trivial, also making $(z^{(1)},\dots,z^{(s)})$ trivial.\\

For $N \leq c_L^n m^n$, the lemma just says $r_L(N) \geq 1$, so let $N > c_L^n m^n$. Take $d$ such that
\[c_L^n m^{dn} \leq N < c_L^n m^{(d+1)n}.\]
This means the $L$-free $Z$ we constructed above sits inside $[M^n] \subseteq [N]$ and has size 
\[|Z|=|X|^d = m^{\alpha nd} \geq \frac1{(c_Lm)^{\alpha n}} N^{\alpha}.\]
\end{proof}

\begin{proof}[Proof of \Cref{theorem:main}]
    Applying \Cref{lemma:L-seed} and \Cref{lemma:project_to_integers} gives 
    \[r_L(N) \geq \frac1{(c_Lm)^{\alpha r\ell}} N^{\alpha}, \quad \text{with }  \alpha = \frac1r\log_m|A| - \frac{s-1}{r\ell}\log_m\left(1+2\ell K\right),\]
    where $K$ is defined as in \Cref{lemma:L-seed}.
    This can be rewritten as
    \[r_L(N) \geq N^{\frac1r\log_m|A|} \exp\left(-O_{A,L}\left(\ell+\frac{\log \ell}\ell\log N\right)\right)\]
    so that for the choice $\ell = \left\lceil \sqrt{\log N \log \log N}\right\rceil$ we indeed get the desired lower bound.
\end{proof}

\section{The peeling method and examples}\label{section:peeling}

A special case of \Cref{theorem:main} also follows from the peeling method. Variants of this method have been used in \cite{karoli_solymosi_peeling} and \cite{improving_behrend}, and in fact all the examples we provide in this paper are peelable. Cases in which our best lower bounds do not come from the peeling method will be explored in forthcoming work.

\medskip

Let $G$ be a finite abelian group, let $A \subseteq G$, and let $k \in [s]$. An ordering $a_1, \dots, a_n$ of the elements of $A$ is called a \emph{peeling order for $x_k$ in $L$} if 
\[(x_1,\dots,x_s) \in \{a_i,a_{i+1},\dots,a_n\}^s, \quad L(x_1,\dots, x_s) = 0, \quad x_k = a_i \quad\implies \quad a_i = x_1 = x_2 = \dots = x_s.\]
In words, this means that in $A$, every solution with $a_1$ in the $k$th position must be diagonal. If we remove $a_1$, then every solution in $A \setminus \{a_1\}$ with $a_2$ in the $k$th position must be diagonal, and so forth. If there is a peeling order of $A$, we say $A$ is \emph{peelable}. If $L$ has genus at least two, the only peelable sets are singletons, so the peeling method gives nontrivial lower bounds only for genus one forms. The following proposition can be used in place of \Cref{lemma:L-seed} to prove \Cref{theorem:main}. We include a short proof to see the Salem--Spencer method \cite{salem1942sets} in action.

\begin{proposition}[The peeling method]\label{lemma:peeling_method}
    Let $L$ be of genus one and let $A \subseteq G$ be peelable for $L$ with peeling order $a_1,\dots,a_n$ for $x_1$. Then for every $\ell$ divisible by $n$ there exists an $L$-free set $X \subseteq G^\ell$ of size 
    \[|X| = \frac{\ell!}{((\ell/n)!)^n} \geq n^{\ell-o(\ell)}.\]
\end{proposition}
\begin{proof}
Let $X \subseteq A^\ell$ be the set of all tuples $(x_1,\dots,x_\ell)$ in which each $a \in A$ appears exactly $\ell/n$ times. It has the claimed size by Stirling's approximation. Suppose there is a nontrivial solution $(x^{(1)},\dots, x^{(s)}) \in X^s$ to $L=0$ and write $x^{(i)} = (x^{(i)}_1,\dots,x^{(i)}_\ell)$. Consider all those $x^{(i)}_j$ for which the $j$-th coordinate $(x^{(1)}_j,\dots, x^{(s)}_j)$ is not a trivial solution. Let $k \in [n]$ be minimal so that $a_k \in A$ is among those $x^{(i)}_j$. By the peeling property, $x^{(1)}$ cannot have $a_k$ as an entry in a coordinate $j \in [\ell]$ in which the coordinate-wise solution is nontrivial. This shows that $x^{(1)}$ contains strictly fewer occurrences of $a_k$ than another $x^{(i)}$, which contradicts the choice of $X$. 
\end{proof}

\begin{proposition}
    Let $L$ be of genus one and let $A \subseteq G$ be peelable. Then $A$ is an $L$-seed. 
\end{proposition}
\begin{proof}
    Let $a_1,\dots,a_n$ be a peeling order of $A$ for $x_1$. Define $f: A \to \Z$ by $f(a_\ell) = s^{n-\ell}$ and let $f_1 = -(s-1) f$, $f_2=\dots=f_s = f$. We claim that $f_1,\dots, f_s$ are certificate functions for $A$. Indeed, since the only trivial solutions of a genus one equation are diagonal, and since $f_1 + \dots + f_s = 0$, we immediately have 
    \[\sum_{i = 1}^{s} f_i(x_i) = 0 \quad \text{for all trivial solutions } (x_1,\dots,x_s) \in A^s.\]
    Now consider a nontrivial solution $(x_1,\dots, x_s) \in A^s$. Because of the peeling property, we know that if $x_1 = a_\ell$, then there must be a $j \in \{2,3,\dots,s\}$ such that $x_j = a_k$ with $1 \leq k < \ell$. We thus have
    \[\sum_{i = 1}^{s} f_i(x_i) \geq f_1(a_\ell) + f_j(a_k) = -(s-1) s^{n-\ell} + s^{n-k}> 0.\qedhere\] 
    \end{proof}

\subsection{Behrend's construction in the language of $L$-seeds}

Behrend's construction \cite{Behrend} can also be explained using our framework of $L$-seeds. Assume that $L$ is convex, that is, all of the coefficients of $L(x_1, \dots,x_s ) = c_1x_1+ \dots + c_s x_s$ have the same sign, except for exactly one, say $c_1 > 0$ and $c_2,\dots, c_s < 0$. Then $A= \{1,\dots,m\} \subseteq \Z_{c_1m}$ is an $L$-seed with certificate functions 
\[f_1,\dots, f_s: A \to \Z, \quad f_i(x) = -c_ix^2.\] 
Indeed, every nontrivial solution $(x_1,\dots, x_s) \in A^s$ to $L = 0$ mod $c_1m$ is also a solution over $\Z$, so that we get 
\[\sum_{i =1}^s f_i(x_i) = -\sum_{i = 2}^s c_i (x_i -x_1)^2 > 0.\]
 Trivial solutions are again diagonal, and for them we know $f_1 + \dots + f_s = 0$. Applying \Cref{theorem:main} gives $r_L(N) \geq N^{1-\log(c_1)/\log(c_1m) -o(1)}$ and, letting $m \to \infty$, we get $r_L(N) \geq N^{1-o(1)}$.

\medskip

Note that if we alter the coefficients of $L$ by multiples of $c_1m$ in the above, we can construct non-convex forms for which the exponent is arbitrarily close to $1$. Take, for example, the form
\[L_m(x,y,z,w) = (3m+1)x + y - (3m-1) z - 3w.\]
The argument from above still gives $r_{L_m}(N) \geq N^{1-\log(3)/\log(3m) -o(1)}$, since mod $3m$ the form $L_m$ reduces to $x+y+z-3w$. For any $\varepsilon > 0$, we can therefore find a large enough $m$ such that for the non-convex $L_m$ we get $r_{L_m}(N) = \Omega(N^{1-\varepsilon})$. This is Ruzsa's construction and can be found in Theorem 7.5 in \cite{RuzsaSolving}.

\subsection{A human-checkable finite $L$-seed} Consider the form 
\[L_7(x,y,z,w) = x+y+z-3w.\]
We claim that $A = \{0,1,2\} \subseteq \Z_7$ is an $L_7$-seed. In fact, the only nontrivial solutions $(x,y,z,w) \in A^4$ to $L_7 = 0$ are the ones with $\{x,y,z\} = A$ and $w = 1$, so that $A$ is even peelable with peeling order $0,1,2$ for the variable $w$. As a consequence, we see that any translation-invariant form $L$ that reduces to a nonzero multiple of $L_7$ mod $7$ satisfies 
\begin{align}\label{align:012_seed}
    r_{L}(N) \geq N^{\log_7(3) - o(1)}, \quad \log_7(3) > 0.564575.
\end{align}
For example, $L$ can be any of
\[8x+y-6z-3w, \quad 3x+3y-4z-2w, \quad 4x +4y -3z-5w.\]
An exponent for $8x+y-6z-3w$ strictly exceeding $2/3$ would have consequences in ergodic theory; see Section 5.5 in \cite{frantzikinakis}. We were able to obtain a slightly better exponent than the one in \Cref{align:012_seed}, but pushing beyond $2/3$ seems out of reach for us. 

\subsection{$3x+y=2z+2w$ and the $15$-point seed in $\F_5^3$}\label{subsection:15PointSeed} Let $L_3(x,y,z,w) = 3x+y-2z-2w$. A search with AlphaEvolve \cite{novikov2025alphaevolve} gave an $L_3$-seed $A \subseteq \F_5^3$ with $15$ points. Its certificate functions are $f_1,f_2,f_3,f_4 : A \to \Z$ with $f_1 = f_3 = f_4 = f$ and $f_2 = -3f$, where $f$ and $A$ are given in the following table. In fact, the order in which $A$ is listed below (reading from top to bottom first, then from left to right) is a peeling order when we peel the variable $y$.
\[
\begin{array}{c|c}
a & f(a) \\ \hline
(1,0,1) & 183 \\
(0,1,1) & 61 \\
(3,3,2) & 23 \\
(0,3,4) & 8 \\
(1,1,4) & 8
\end{array}
\qquad
\begin{array}{c|c}
a & f(a) \\ \hline
(4,4,2) & 2 \\
(2,0,0) & 3 \\
(1,4,0) & 2 \\
(1,1,3) & 3 \\
(4,0,1) & 1
\end{array}
\qquad
\begin{array}{c|c}
a & f(a) \\ \hline
(0,1,2) & 1 \\
(3,3,3) & 0 \\
(2,0,3) & 1 \\
(0,2,2) & 2 \\
(0,3,1) & 0
\end{array}
\]

Note that together with \Cref{theorem:main}, this certificate gives the exponent in \Cref{theorem:3x+y=2z+2w}. The exponent $\log15/(3\log 5)$ is not optimal, and we were able to find slightly better $L$-seeds in higher dimensions. We present this rather simple construction to illustrate the method. 

\section{Ruzsa's parabola construction and beyond}\label{section:parabola_and_beyond}

The starting point for this section is Theorem 7.3 in \cite{RuzsaSolving}. Consider the translation-invariant linear form 
\begin{align}\label{align:four_vars_form2}
    L_4(x,y,z,w) = ax+by-cz-dw,
\end{align}
where $a,b,c,d > 0$ and $a+b = c+d$. Let $p$ be a prime not dividing $a+b$ such that $abcd$ is a quadratic non-residue mod $p$. Then the parabola 
\[P = \{(x,x^2) : x \in \F_p\} \subseteq \F_p^2\]
is $L_4$-free. Indeed, suppose there is a solution, so there exist $x,y,z,w \in \F_p$ with 
\[ax + by = cz + dw, \quad ax^2 + by^2 = cz^2 + dw^2.\]
Squaring the first equation and subtracting $(a+b)$ times the second gives 
\[ab(x-y)^2 = cd (z-w)^2.\]
Since $abcd$ is not a square in $\F_p$, this forces $x =y$ and $z=w$. The first equation now gives $(a+b)(x-z) = 0$ and hence $x=z$, since $a+b$ is not divisible by $p$.

\medskip

Note that every nonsquare integer is a quadratic non-residue modulo infinitely many primes, so that from \Cref{lemma:project_to_integers} it follows that $r_{L_4}(N) = \Omega(\sqrt N)$ whenever $abcd$ is not a square. When $abcd$ is a square, Ruzsa used Behrend's construction to find a large $L_4$-free subset of the parabola and obtained $r_{L_4}(N) \geq N^{1/2 -o(1)}$. 

\subsection{Beating products for four-variable forms} Applying the following theorem of the first author and Tyrrell to the parabola above immediately shows that, when $abcd$ is not a square, the exponent $1/2$ is never optimal. 

\begin{theorem}[Corollary 1.4 in \cite{beatingProducts}]\label{theorem:beatingProducts}
Let $L(x_1,\ldots,x_s)=0$ be a genus one translation-invariant equation over $\F_q$, where $s \geq 3$ and $q$ is a prime power. If $A\subseteq \F_q^n$ is $L$-free and of size $|A| = c^n$, then there is an integer $m>n$ and an $L$-free set $B\subseteq \F_q^m$ of size $|B|>c^m$.
\end{theorem}

Note that in the four-variable case above, $abcd$ not being a square implies that $L_4$ is of genus one. Together with the parabola construction, this implies \Cref{theorem:four_vars_form}. 

\fourvars*

More generally, \Cref{lemma:project_to_integers} and \Cref{theorem:beatingProducts} show that the exponent obtained from an $L$-free set in $\F_p^n$ is never optimal if $L$ has genus one mod $p$. For higher genus, we cannot hope for \Cref{theorem:beatingProducts} to still hold in general. It could still be the case, however, that at least when $q$ is prime, it holds. 

\medskip 

Opening up the proof of \Cref{theorem:beatingProducts} shows that the $\varepsilon_L$ in \Cref{theorem:four_vars_form} is tiny. For example, for $L_3(x,y,z,w) = 3x+y-2z-2w$, a naive application improves the exponent $1/2$ roughly in the 180th decimal place. 

\subsection{The projective completion of the moment curve} 

For a general translation-invariant form $L(x_1,\dots,x_s)= c_1 x_1 + \dots + c_s x_s$, Theorem 2.1 in \cite{RuzsaSolving} uses a greedy construction to show that $r_L(N) = \Omega_s(N^{1/(s-1)})$. For primitive $L$, the same exponent can be recovered using a moment curve together with \Cref{lemma:project_to_integers}.

\begin{proposition}\label{proposition:usual_moment_curve}
   Let $p$ be a prime such that $L$ has the same coefficient partition over $\F_p$ and over $\Z$. Then the $(s-1)$-dimensional moment curve mod $p$, that is,
\[\Gamma = \{ (t, t^2, \dots, t^{s-1}) : t \in \F_p\} \subseteq \F_p^{s-1},\]
is $L$-free.
\end{proposition}

\begin{proof}
    Suppose there is a solution; that is, there exist $t_1,\dots, t_s \in \F_p$ such that 
\[\sum_{i = 1}^s c_i t_i^k = 0, \quad \text{for all } 0 \leq k \leq s-1.\]
For $k=0$, we used the translation invariance of $L$ and the convention that $0^0 = 1$. Let $\mathcal T$ be the partition of $[s]$ such that $i,j \in [s]$ belong to the same block of $\mathcal T$ iff $t_i = t_j$. For each $T \in \mathcal T$, let $c_T = \sum_{i \in T} c_i$ and $t_T = t_i$, where $i \in T$, so that we get 
\[\sum_{T \in \mathcal T} c_T t_T^k = 0 \mod p, \quad \text{for all } 0 \leq k \leq s-1.\]
Since the $t_T$ are distinct and $|\mathcal T| \leq s$, this Vandermonde system only has the trivial solution $c_T \equiv 0$ mod $p$ for all $T \in \mathcal T$. By the choice of $p$, this shows that our original solution to $L=0$ is trivial. 
\end{proof}

The moment-curve construction can be improved by taking its projective completion and adding the point at infinity. 

\begin{theorem}
Let $p$ be a prime such that $L$ has the same coefficient partition over $\F_p$ and over $\Z$. Then there is an $L$-free set in $\F_p^{s-1}$ of size $p+1$. 
\end{theorem}
\begin{proof}
Let 
\[\overline \Gamma = \{(1,t,\dots,t^{s-1}) : t \in \F_p\} \cup \{(0,\dots,0,1)\}\subseteq \F_p^s.\]
Again, using the invertibility of Vandermonde matrices, it follows that every set of at most $s$ elements of $\overline \Gamma$ is linearly independent. Let 
\[f(x) = a_0 + a_1 x + \dots + a_{s-1} x^{s-1} \in \F_p[x]\]
be an irreducible polynomial over $\F_p$ of degree $s-1$ and consider the set 
\[S = \{(1,t,\dots,t^{s-1})/f(t) : t \in \F_p\} \cup \{(0,\dots,0,1)/a_{s-1}\}.\]
By definition, $S$ lies in the affine hyperplane $H\subseteq \F_p^s$ given by $\sum_{i =0}^{s-1} a_{i} x_i = 1$ and still has the property that every subset of at most $s$ elements is linearly independent. Much like in the proof of \Cref{proposition:usual_moment_curve}, the latter property implies that $S$ is $L$-free. Identifying $H$ with $\F_p^{s-1}$ proves the theorem.
\end{proof}

For $3 \leq s \leq p$, a theorem of Ball \cite{BallArcs} states that there are no $p+2$ points in $\F_p^s$ with the above property and that, up to a change of basis, $\overline \Gamma$ is the unique such set with $p+1$ elements. This implies the prime field case of the famous MDS conjecture, which first appeared as a question by Segre in \cite{segre1955curve}.

\medskip

Let $L$ be primitive and let $p$ be a prime for which $L$ has the same coefficient partition over $\F_p$ and over $\Z$. Together with \Cref{lemma:project_to_integers}, this theorem implies 
\[r_L(N) = \Omega_{L}\left( N^{\frac{\log(p+1)}{(s-1)\log p}}\right),\]
so, in particular, \Cref{theorem:proj_compl_moment} follows. 

\subsection{The moment curve in one dimension lower}

Ruzsa's parabola construction shows that sometimes the moment curve in $\F_p^{s-2}$ is already $L$-free, which gives the lower bound $\Omega(N^{1/(s-2)})$ for primitive $L$. Recall that the parabola mod $p$ is $L_4$-free for $L_4(x,y,z,w) = ax + by - cz - dw$ whenever $abcd$ is not a square mod $p$. There is a direct generalization of this sufficient condition for even $s$. 

\begin{theorem}
    Let $s$ be even and let $p$ be a prime such that $L$ has the same coefficient partition over $\F_p$ and over $\Z$. Assume that 
    \begin{align}\label{align:product_cond}
        (-1)^{s/2} \prod_{i = 1}^s c_i
    \end{align}
    is not a quadratic residue mod $p$. Then $\Gamma = \left\{(t,t^2, \dots, t^{s-2}) : t \in \F_p\right\} \subseteq \F_p^{s-2}$ is $L$-free.
\end{theorem}

\begin{proof}
    Suppose there is a solution; that is, there exist $t_1,\dots, t_s \in \F_p$ with 
    \begin{align}\label{align:reducedVandermonde}
        \sum_{i = 1}^s c_i t_i^k = 0, \quad \text{for all } 0 \leq k \leq s-2.
    \end{align}
    If $t_1,\dots,t_s$ are not pairwise distinct, then the same argument as in \Cref{proposition:usual_moment_curve} shows that the solution is trivial, so let us assume that $t_1,\dots,t_s$ are pairwise distinct. Inspecting the one-dimensional kernel of the system in \Cref{align:reducedVandermonde} shows that
    \[c_i = \lambda\left(\prod_{j \neq i} (t_i-t_j)\right)^{-1},\]
    for some $\lambda \in \F_p \setminus \{0\}$. Taking the product over $i \in [s]$ gives 
    \[\prod_{i =1}^s c_i =\lambda^s (-1)^{s \choose 2} \prod_{i < j} (t_i-t_j)^{-2},\]
    which contradicts the hypothesis on \Cref{align:product_cond} since $s$ is even. 
\end{proof}

For every nonsquare integer $n$, there are infinitely many primes $p$ such that $n$ is not a square in $\F_p$. Thus \Cref{theorem:reduced_moment_even_s} follows, that is, if the quantity in \Cref{align:product_cond} is not a square and $L$ is primitive, then 
\[r_L(N) = \Omega_{L}( N^{\frac1{s-2}}).\]

\section*{Computational and AI Assistance}
For the computational searches and numerical experiments, the authors used AlphaEvolve, an evolutionary coding agent powered by Google DeepMind's Gemini models \cite{novikov2025alphaevolve}. The authors also used OpenAI's ChatGPT, primarily for assistance with literature searches, checking calculations, and proofreading the paper. The mathematical ideas, methods, and proofs are due to the authors. The authors independently verified all computer-generated examples and certificates used in the paper and take full responsibility for the mathematical content.

\section*{Acknowledgements}

The second author's research was supported in part by an NSERC Discovery Grant and by the National Research, Development, and Innovation Office of Hungary, NKFIH, Grant No.~KKP133819 and Excellence 151341. We thank Yuval Wigderson for carefully reading the manuscript and for helpful comments. Shortly before submission of this paper, we learned that Bhavik Mehta independently obtained a similar lower bound as the one in \Cref{theorem:3x+y=2z+2w}.

\printbibliography

@article{BallArcs,
  author       = {Ball, S.},
  title        = {On sets of vectors of a finite vector space in which every subset of basis size is a basis},
  journal      = {Journal of the European Mathematical Society},
  year         = {2012},
  volume       = {14},
  number       = {3},
  pages        = {733--748},
  doi          = {10.4171/jems/316}
}

@article{Behrend,
  author       = {Behrend, F. A.},
  title        = {On sets of integers which contain no three terms in arithmetical progression},
  journal      = {Proceedings of the National Academy of Sciences of the United States of America},
  year         = {1946},
  volume       = {32},
  number       = {12},
  pages        = {331--332},
  doi          = {10.1073/pnas.32.12.331}
}

@book{algebraicComplTheory,
  author       = {B{\"u}rgisser, P. and Clausen, M. and Shokrollahi, M. A.},
  title        = {Algebraic complexity theory},
  year         = {1997},
  series       = {Grundlehren der mathematischen Wissenschaften},
  volume       = {315},
  publisher    = {Springer-Verlag},
  location     = {Berlin and Heidelberg},
  isbn         = {978-3-540-60582-9},
  doi          = {10.1007/978-3-662-03338-8}
}

@inproceedings{degenerations_corners,
  author       = {Christandl, M. and Fawzi, O. and Ta, H. and Zuiddam, J.},
  title        = {Larger corner-free sets from combinatorial degenerations},
  booktitle    = {13th Innovations in Theoretical Computer Science Conference ({ITCS} 2022)},
  year         = {2022},
  pages        = {48:1--48:20},
  doi          = {10.4230/LIPIcs.ITCS.2022.48}
}

@misc{improving_behrend,
  author       = {Elsholtz, C. and Hunter, Z. and Proske, L. and Sauermann, L.},
  title        = {Improving {Behrend}'s construction: Sets without arithmetic progressions in integers and over finite fields},
  year         = {2024},
  eprint       = {2406.12290},
  eprinttype   = {arxiv},
  eprintclass  = {math.NT}
}

@misc{frantzikinakis,
  author       = {Frantzikinakis, N.},
  title        = {Multiple ergodic averages for three polynomials and applications},
  year         = {2007},
  eprint       = {math/0606567},
  eprinttype   = {arxiv},
  eprintclass  = {math.DS}
}

@online{green100problems,
  author       = {Green, B.},
  title        = {100 open problems},
  year         = {2026},
  url          = {https://people.maths.ox.ac.uk/greenbj/papers/open-problems.pdf},
  urldate      = {2026-09-11}
}

@misc{beatingProducts,
  author       = {Hametner, P. and Tyrrell, F.},
  title        = {Beating product constructions for linear equations over finite fields},
  year         = {2026},
  eprint       = {2606.12194},
  eprinttype   = {arxiv},
  eprintclass  = {math.CO}
}

@misc{karoli_solymosi_peeling,
  author       = {K\'arolyi, G. and Solymosi, J.},
  title        = {A {Salem--Spencer}-type construction for large subsets of integer grids with no isosceles right triangles},
  year         = {2026},
  eprint       = {2607.22828},
  eprinttype   = {arxiv},
  eprintclass  = {math.CO}
}

@article{novikov2025alphaevolve,
  title   = {{AlphaEvolve}: A coding agent for scientific and algorithmic discovery},
  author  = {Novikov, Alexander and V{\~u}, Ng{\^a}n and Eisenberger, Marvin and Dupont, Emilien and Huang, Po-Sen and Wagner, Adam Zsolt and Shirobokov, Sergey and Kozlovskii, Borislav and Ruiz, Francisco J. R. and Mehrabian, Abbas and Kumar, M. Pawan},
  journal = {arXiv preprint arXiv:2506.13131},
  year    = {2025}
}

@article{RuzsaSolving,
  author       = {Ruzsa, I. Z.},
  title        = {Solving a linear equation in a set of integers. {I}},
  journal      = {Acta Arithmetica},
  year         = {1993},
  volume       = {65},
  number       = {3},
  pages        = {259--282},
  doi          = {10.4064/aa-65-3-259-282}
}

@article{salem1942sets,
  author       = {Salem, R. and Spencer, D. C.},
  title        = {On sets of integers which contain no three terms in arithmetical progression},
  journal      = {Proceedings of the National Academy of Sciences of the United States of America},
  year         = {1942},
  volume       = {28},
  number       = {12},
  pages        = {561--563}
}

@article{segre1955curve,
  author       = {Segre, B.},
  title        = {Curve razionali normali e {$k$}-archi negli spazi finiti},
  journal      = {Annali di Matematica Pura ed Applicata},
  year         = {1955},
  volume       = {39},
  pages        = {357--379}
}

@article{Strassen1987,
  author       = {Strassen, V.},
  title        = {Relative bilinear complexity and matrix multiplication},
  journal      = {Journal f{\"u}r die reine und angewandte Mathematik},
  year         = {1987},
  volume       = {375/376},
  pages        = {406--443},
  doi          = {10.1515/crll.1987.375-376.406}
}
\end{document}